\documentclass[11pt]{article}

\usepackage[top=1in, left=1in, bottom=1.5in, right=1.5in]{geometry}

\usepackage{hyperref}  
\usepackage{setspace} 

\usepackage{graphicx}
\usepackage{amsfonts}
\usepackage{amssymb}
\usepackage{amsmath}
\usepackage{amsthm}

\newtheorem{theorem}{Theorem}[section]

\newtheorem{lemma}[theorem]{Lemma}
\newtheorem{conjecture}[theorem]{Conjecture}
\newtheorem{claim}{}[theorem]

\newcommand{\del}{\backslash}

\title{Graphs with girth $2\ell$ and without longer even holes are $3$-colorable}

\author{Rong Chen\\
\\
Center for Discrete Mathematics,\ \ Fuzhou University\\
Fuzhou,\ \ P. R. China}

\begin{document}

\maketitle

\footnote{

Email: rongchen@fzu.edu.cn
}

\begin{abstract}
For a number $\ell\geq 2$, let $\mathcal{H}_{\ell}$ denote the family of graphs which have girth $2\ell$ and have no even hole with length greater than $2\ell$. Wu, Xu, and Xu conjectured that every graph in $\bigcup_{\ell\geq2}\mathcal{H}_{\ell}$ is 3-colorable. In this paper, we prove that every graph in $\mathcal{H}_{\ell}$ is 3-colorable for any integer $\ell\geq5$.

\end{abstract}

\section{Introduction}

All graphs considered in this paper are finite, simple, and undirected.
A {\em proper coloring} of a graph $G$ is an assignment of colors to the vertices of $G$ such that no two adjacent vertices receive the same color. A graph is {\em $k$-colorable} if it has a proper coloring using at most $k$ colors. The {\em chromatic number} of $G$, denoted by $\chi(G)$, is the minimum number $k$ such that $G$ is $k$-colorable.

A {\em hole} in a graph is an induced cycle of length at least four. A hole is {\em even} (resp. {\em odd}) if it has even (resp. odd) length. The Strong Perfect Graph Theorem  asserts that a graph is perfect if and only if it induces neither odd holes nor their complements. For any odd-hole-free graph $G$, we have $\chi(G)\leq 2^{2^{\omega(G)+2}}$ by the main result proved by Scott and Seymour in \cite{AS16}, while Ho\'ang \cite{AS} conjectured that $\chi(G)\leq \omega(G)^2$. For any even-hole-free graph, Addario-Berry, Chudnovsky, Havet, Reed and Seymour \cite{AC25}, and Chudnovsky and Seymour \cite{MC23}, proved that every even-hole-free graph has a vertex whose neighbours are the union of two cliques, which implies $\chi(G)\leq 2\omega(G)-1$.

The {\em girth} of a graph $G$, denoted by $g(G)$, is the minimum length of a cycle in $G$. 
For any integer $\ell\geq2$, let $\mathcal{G}_{\ell}$ be the family of graphs that have girth $2\ell + 1$ and have no odd holes of length at least $2\ell + 3$, and let $\mathcal{H}_{\ell}$ be the family of graphs that have girth $2\ell$ and have no even holes of length at least $2\ell + 2$.
Robertson conjectured in \cite{ND11} that the Petersen graph is  the only graph in $\mathcal{G}_2$ that is 3-connected and internally 4-connected.
Plummer and Zha \cite{PM14} disproved Robertson's conjecture and proposed 
the strong conjecture that all graphs in $\mathcal{G}_2$ are 3-colorable, which was solved by Chudnovsky and Seymour \cite{MC22}. 
Generalizing the result of \cite{XB17}, Wu, Xu, and Xu \cite{WD2204} 
proposed the following conjecture.

\begin{conjecture}\label{conj}(\cite{WD2204}, Conjecture 6.1.)
Every graph in $\bigcup_{\ell\geq2}\mathcal{G}_{\ell}$ is $3$-colorable.
\end{conjecture}
Wu, Xu and Xu \cite{WD22} proved that Conjecture \ref{conj} holds for $\ell=3$. Recently, Chen \cite{C25} proved that Conjecture \ref{conj} holds for $\ell\geq5$. 
More recently, Wang and Wu \cite{WW23} further proved that Conjecture \ref{conj} holds for $\ell=4$, and finally solve Conjecture \ref{conj}. 
Whether are all graphs in $\mathcal{H}_{\ell}$ also 3-colorable? Wu, Xu, and Xu 
conjectured

\begin{conjecture}\label{conj-e}
Every graph in $\bigcup_{\ell\geq2}\mathcal{H}_{\ell}$ is $3$-colorable.
\end{conjecture}

In this paper, we prove that Conjecture \ref{conj-e} holds for all integers $\ell\geq5$. That is,
\begin{theorem}\label{main thm}
For each integer $\ell\geq5$, all graphs in $\mathcal{H}_{\ell}$ are $3$-colorable.
\end{theorem}

In fact, we prove a little stronger result than Theorem \ref{main thm}.
\begin{theorem}\label{main thm+}
For each integer $\ell\geq5$, all graphs in $\mathcal{H}_{\ell}$ have either a degree-$2$ vertex or a $K_1$- or $K_2$-cut.
\end{theorem}
\begin{proof}[ Proof of Theorem \ref{main thm} assuming Theorem \ref{main thm+}.]
Assume not. Let $G\in \mathcal{H}_{\ell}$ be a minimal counterexample to Theorem \ref{main thm}. Then $\delta(G)\geq3$. It follows from Theorem \ref{main thm+} that $G$ has a $K_1$- or $K_2$-cut. Then there are two induced subgraphs $G_1,G_2$ of $G$ with $G=G_1\cup G_2$ and such that $G_1\cap G_2$ is a vertex or an edge. Since $G_1,G_2$ are 3-colorable by the choice of $G$, so is $G$, a contradiction. 
\end{proof}
Since all complete bipartite graphs are in $\mathcal{H}_2$, the methods used in this paper can not be generalized to solve Conjecture \ref{conj-e} holding for $\ell=2$. 
But the methods probably can be generalized to work for the the case $\ell=4$ as the condition $\ell\geq5$ is only used in the last proof of this paper.


\section{Preliminaries}
A {\em cycle} is a connected $2$-regular graph. Let $P$ be an $(x,y)$-path and $Q$ a $(y,z)$-path.
When $P$ and $Q$ are internally disjoint paths, 
let $PQ$ denote the $(x,z)$-path $P\cup Q$. Evidently, $PQ$ is a path when $x\neq z$, and $PQ$ is a cycle when $x=z$. Let $P^*$ denote the set of internal vertices of $P$. When $u,v\in V(P)$, let $P(u,v)$ denote the sub-path of $P$ with ends $u, v$. For simplicity, we will let $P^*(u,v)$ denote $(P(u,v))^*$.

Let $G$ be a graph. We say a path $P$ of $G$ is an {\em ear} of $G$ if all vertices in $P^*$ have degree-2 in $G$ while both ends of $P$ have degree at least 3. For any subset $U\subseteq V(G)$, let $G[U]$ be the induced subgraph of $G$ defined on $U$. For subgraphs $H$ and $H'$ of $G$, set $|H|:=|E(H)|$ and $H\Delta H':=E(H)\Delta E(H')$.
Let $H\cup H'$ denote the subgraph of $G$ with vertex set $V(H)\cup V(H')$ and edge set $E(H)\cup E(H')$. Let $H\cap H'$ denote the subgraph of $G$ with edge set $E(H)\cap E(H')$ and without isolated vertex.
Let $N(H)$ be the set of vertices in $V(G)-V(H)$ that have a neighbour in $V(H)$. Set $N[H]:=N(H)\cup V(H)$. Let $A_i$ be a subgraph of $G$ or a subset of $V(G)$ for any integer $1\leq i\leq2$ with $V(A_1)\cap V(A_2)=\emptyset$. If there is no edge of $G$ linking $A_1$ and $A_2$, we say that $A_1,A_2$ are {\em anti-complete}. We say a path $P$ is an {\em $(A_1,A_2)$-path} if $V(A_1\cup A_2)\cap V(P^*)=\emptyset$ and 
 one end of $P$ is in $V(A_1)$ and the other end is in $V(A_2)$.
For any subgraph $H$ of $G$ containing $V(A_1\cup A_2)$, let $d_H(A_1,A_2)$ denote the minimum length of all $(A_1,A_2)$-paths in $H$.
\begin{lemma}\label{induced tree}
Let $S$ be a set of degree-$1$ vertices of a connected graph $G$. If $g(G)>|S|$, then there is an induced tree of $G$ containing $S$.
\end{lemma}
\begin{proof}
Let $T$ be an induced connected subgraph of $G$ containing $S$ with $|V(T)|$ minimal. Assume that $T$ is not a tree. Then $T$ contains a cycle $C$. Since $S$ is a set of degree-$1$ vertices of $G$, we have $V(C)\cap S=\emptyset$. Since $g(G)>|S|$ implies $|C|>|S|$, there is a vertex $x\in V(C)$ such that $T\del x$ contains a connected subgraph containing $S$, a contradiction to the choice of $T$. So $T$ is a tree.
\end{proof}


A {\em theta graph} is a graph that consists of a pair of distinct vertices joined by three internally disjoint paths.
\begin{lemma}\label{easy case}
Let $\ell\geq 2$ be an integer and $H$ be an induced theta subgraph of a graph  $G\in \mathcal{H}_{\ell}$. \begin{itemize}
\item[(1)] When all cycles in $H$ are even, either all ears of $H$ have length $\ell$ or exactly one ear of $H$ has length one.
\item[(2)] When some cycle in $H$ is odd, the length of the ear of $H$ shared by its odd cycles is one or larger than $\text{max} \{\ell, |P|, |P'|\}$, where $P,P'$ are the ears of $H$ shared by cycles having different parity.
\end{itemize}
\end{lemma}
\begin{proof}
First we consider (1). Since $g(G)=2\ell\geq4$, at most one ear of $H$ has length one. When no ear of $H$ has length one, all cycles in $H$ are even holes with length $2\ell$ as $H$ is induced, so all ears of $H$ have length $\ell$. This proves (1).

Now we consider (2). Since some cycle in $H$ is odd, there are exactly two cycles of $H$ are odd. Assume that the ear of $H$, say $Q$, shared by its odd cycles has length at least two. Then the unique even cycle of $H$ is an even hole with length $2\ell$. Since $g(G)=2\ell$ and $|P|+ |P'|=2\ell$, we have $|Q|>\text{max} \{|P|, |P'|,\ell\}$, where $P,P'$ are the ears of $H$ not equal to $Q$. This proves (2). 
\end{proof}

\begin{lemma}\label{neighbour in C}
Let $\ell\geq 3$ be an integer and $C$ be an even hole of a graph $G\in \mathcal{H}_{\ell}$. Then every vertex not in $C$ has at most one neighbour in $C$.
\end{lemma}
\begin{proof}
Assume to the contrary that a vertex $v\notin V(C)$ has two neighbours $s,t\in V(C)$. Let $Q$ be a $(s,t)$-path on $C$ with length at most $\ell$. Then $Qsvt$ is a cycle of length at most $\ell+2<2\ell$ as $\ell\geq3$, a contradiction.
\end{proof}

Let $C$ be a hole of a graph $G$ and $s,t\in V(C)$ nonadjacent. Let $P$ be an induced $(s,t)$-path. If $V(C)\cap V(P^*)=\emptyset$, we call $P$ a {\em jump} or an {\em $(s,t)$-jump} over $C$. Let $Q_1,Q_2$ be the internally disjoint $(s,t)$-paths of $C$.
If $Q^*_1$ is not anti-complete to $P^*$ while $Q^*_2$ is anti-complete to $P^*$, we say that $P$ is a {\em local jump over $C$ across $Q^*_1$}. When there is no need to strengthen $Q^*_1$, we will also say that $P$ is a {\em local jump over $C$}. In particular, when $|V(Q^*_1)|=1$, we say that $P$ is a {\em local jump over $C$ across one vertex}. When $Q^*_1\cup Q^*_2$ is anti-complete to $P^*$, we say that $P$ is a {\em short jump over $C$}. 

\begin{lemma}\label{find short jump}
Let $\ell\geq3$ be an integer and  $C$ be an even hole of a graph $G\in \mathcal{H}_{\ell}$. For any pair of non-adjacent vertices $s,t\in V(C)$, if $P$ is an $(s,t)$-path with $|V(P^*)\cap N(C)|\leq2$ and $V(P^*)\cap V(C)=\emptyset$, then $G[V(P)]$ contains a short $(s,t)$-jump.
\end{lemma}
\begin{proof}
Let $P'$ be a shortest $(s,t)$-path of $G[V(P)]$. 
Since $|V(P^*)\cap N(C)|\leq2$ and no vertex outside $C$ has two neighbours in $C$ by Lemma \ref{neighbour in C}, $P'$ is a short $(s,t)$-jump.
\end{proof}

\begin{lemma}\label{xyz}
Let $C$ be a hole of a graph $G$ such that no vertex outside $C$ has two neighbours in $C$. 
Assume that  $G$ has neither a degree-$2$ vertex nor a $K_1$- or $K_2$-cut. Then for any $3$-vertex path $xyz$ on $C$, one of the following holds.
\begin{itemize}
\item[(1)] Either $x$ or $z$ is one end of a short  jump over $C$.
\item[(2)] There is a local $(x,z)$-jump over $C$ across $y$.
\item[(3)] The vertex $y$ is an end of a short jump or a local jump over $C$ across one vertex.
\end{itemize}
\end{lemma}
\begin{proof}
First we consider the case that $y$ and $C\del\{x,y,z\}$ are in different components of $G\del\{x,z\}$. 
Let $D$ be the component of $G\del\{x,z\}$ containing $y$. Evidently, $D\del y$ is non-empty as $y$ has at least three neighbours. Since neither $\{x,y\}$ nor $\{y,z\}$ is a cut of $G$, both $x$ and $z$ have a neighbour in each component of $D\del y$. Hence, there is a shortest $(x,z)$-path $P$ with interior in $D\del y$. Since $P$ is a short $(x,z)$-jump or a local $(x,z)$-jump over $C$ across $y$, either (1) or (2) holds.

Secondly, consider the case that $y$ and $C\del\{x,y,z\}$ are in the same component of $G\del\{x,z\}$. 
Let $P$ be a shortest $(y,C\del\{x,y,z\})$-path in $G\del\{x,z\}$. Let $y'$ be the end of $P$ not equal to $y$. Since no vertex outside $C$ has two neighbours in $C$, no vertex in $V(C)-\{x,y,y',z\}$ has a neighbour in $P^*$. Assume that $x$ and $z$ have a neighbour in $P^*$. Then there are $x',z'\in V(P^*)$ with $xx',zz'\in E(G)$ and such that no vertex in $V(P^*(x',z'))$ has a neighbour in $V(C)$, so $xx'P(x',z')z'z$ is a short $(x,z)$-jump over $C$, implying that (1) holds. Hence, by symmetry we may assume that $z$ is anti-complete to $P^*$. We may further assume that $x$ has a neighbour in $P^*$, for otherwise $P$ is a short $(y,y')$-jump over $C$. 
When $xy'\in E(G)$, the path $P$ is a local $(y,y')$-jump across $x$, so (3) holds. When $xy'\notin E(G)$, the path $P$ contains a short $(x,y')$-jump over $C$, so (1) holds.
\end{proof}

When $C$ is an even hole of a graph $G\in \mathcal{H}_{\ell}$ with $\ell\geq3$, since no vertex outside $C$ has two neighbours in $C$ by Lemma \ref{neighbour in C}, each $3$-vertex path $xyz$ on $C$ satisfies Lemma \ref{xyz}. This fact will be frequently and only used in the last proof of this paper. 
Hence, in Section 4, we only consider short jumps or local jumps over $C$ across one vertex over an even hole $C$ of a graph $G\in \mathcal{H}_{\ell}$. 

Let $P$ be a short $(s,t)$-jump over $C$, and $Q_1,Q_2$ be the $(s,t)$-paths of $C$. Note that $C\cup P$ is an induced theta subgraph of $G$ whose ears all have length at least 2, and both $PQ_1$ and $PQ_2$ have the same parity as $C$ is even. 
When $PQ_1$ is odd, 
we say that $P$ is of {\em type-o}; and when $PQ_1$ is even, we say that $P$ is of {\em type-e}. By Lemma \ref{easy case} (2), we have $|P|>\ell$ when $P$ is  of type-o. When $P$ is of type-e, it follows from Lemma \ref{easy case} (1) that $|P|=|Q_1|=|Q_2|=\ell$.  Hence, we have.
\begin{lemma}\label{jump-length}
Let $\ell\geq2$ be an integer and  $C$ be an even hole of a graph in $\mathcal{H}_{\ell}$. Then all short jumps $P$ over $C$ have length at least $\ell$, and $|P|=\ell$ if and only if $P$ is of type-e. Moreover, when $|P|=\ell$, the two paths on $C$ with the same ends as $P$ have length $\ell$.
\end{lemma}
\noindent Lemma \ref{jump-length} will be frequently used in the rest of this paper without reference sometimes.

\section{Induced subgraphs isomorphic to subdivisions of $K_4$}
Let $C$ be a cycle and suppose that $x_1,x_2,\ldots,x_n\in V(C)$ occur on $C$ in this cyclic order with $n\geq3$. For any two distinct $x_i$ and $x_j$, the cycle $C$ contains two $(x_i,x_j)$-paths. Let $C(x_i,x_{i+1},\ldots, x_j)$ denote the $(x_i,x_j)$-path in $C$ containing $x_i,x_{i+1},\ldots, x_j$ (and not containing $x_{j+1}$ if $i \neq j+1$), where subscripts are modulo $n$. Such path is uniquely determined as $n \geq 3$.

\begin{figure}[htbp]
\begin{center}
\includegraphics[height=6cm,page=1]{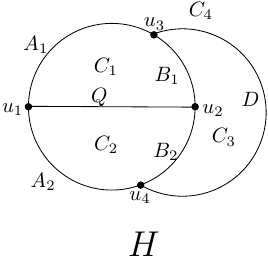}
\caption{$u_1,u_2,u_3,u_4$ are the degree-3 vertices of $H$. Let $C_1,C_2,C_3,C_4$ be the face cycles of $H$. Let $a_1, a_2, b_1, b_2,d,q$ denote the lengths of ears $A_1, A_2, B_1, B_2,D,Q$, respectively. Set $a:=a_1+ a_2$ and $b:=b_1+ b_2$. }
\label{k4}
\end{center}
\end{figure}
Let $H$ be an induced subgraph of a graph $G$ that is isomorphic to a subdivision of $K_4$. When all face cycles of $H$ are odd, we say that $H$ is an {\em odd $K_4$-subdivision}; and when all face cycles of $H$ are even, we say that $H$ is an {\em even $K_4$-subdivision}; and when exactly two face cycles of $H$ are even, 
we say that $H$ is a {\em balanced $K_4$-subdivision}.

\begin{lemma}\label{odd K4}
Let $\ell\geq2$ be an integer and $G$ be a graph in $\mathcal{H}_{\ell}$. Let $H$ be an odd $K_4$-subdivision of $G$ pictured as the graph in Figure \ref{k4}. Then there are two ears of $H$ sharing an end having length one.
\end{lemma}
\begin{proof}
Assume not. By symmetry we may assume that $q\geq\text{max} \{a_1,b_1\}$. By the symmetry between $a_2$ and $b_2$, we may further assume that $a_2>1$. Then $C_1\cup C_3$ is an induced theta-subgraph of $G$. Applying Lemma \ref{easy case} (2) on $C_1\cup C_3$, we have $b_1=1$. Then $a_1, b_2>1$, for otherwise the lemma holds. Since $C_2\Delta C_3$ is an even hole, $q+a_2+d = 2\ell-1$, so $|C_4|=a_1+a_2+d\leq2\ell-1$ as $q\geq a_1$, a contradiction to the fact $g(G)=2\ell$.
\end{proof}

\begin{lemma}\label{even K4}
Let $\ell\geq3$ be an integer and $G$ be a graph in $\mathcal{H}_{\ell}$. Then $G$ does not contain an even $K_4$-subdivision.
\end{lemma}
\begin{proof}
Assume to the contrary that $H$ is an even $K_4$-subdivision of $G$ pictured as the graph in Figure \ref{k4}. Since each even hole of $G$ has length $2\ell$, by symmetry we may assume that $d>1$. When $q=1$, since $a+b=2(2\ell-1)$, by symmetry we may assume that $b\geq 2\ell-1$, so $|C_3|>2\ell$ as $d>1$, a contradiction to the fact $g(G)=2\ell$. Hence, $q>1$. Since $H\del V(Q^*)$ and $H\del V(D^*)$ are induced theta-subgraph of $G$ with all cycles are even, by Lemma \ref{easy case} (1), 
$q=a_1+b_1=a_2+b_2=\ell$ and 
$d=a_1+a_2=b_1+b_2=\ell$. So $a_1=b_2$ and $a_2=b_1$. Since $a_1+b_1=\ell\geq3$, by symmetry assume $a_2=b_1>1$. Then $C_1\Delta  C_3$ is an even hole of length larger than $2\ell$ as $q=d=\ell$, a contradiction.
\end{proof}

When $H$ an induced subgraph of a graph $G\in \mathcal{H}_{\ell}$ with $\ell\geq3$ that is isomorphic to a subdivision of $K_4$, if some face cycle of $H$ is even, then $H$ is balanced by Lemma \ref{even K4}. This fact will be frequently used in the rest of this paper without explanation sometimes.


\begin{lemma}\label{bal K4}
Let $\ell\geq2$ be an integer and $G$ be a graph in $\mathcal{H}_{\ell}$. Let $H$ be a balanced $K_4$-subdivision of $G$ pictured as the graph in Figure \ref{k4} such that $C_1, C_2$ are odd holes and $C_3, C_4$ are even holes. 
Then the following hold.
\begin{itemize}
\item[(1)] $q>1$.
\item[(2)] When $d>1$, we have that $q>\ell$, $a=b=d=\ell$. Moreover, when $\ell>3$,  one of the following holds.
\begin{itemize}
\item[(2.1)] The graph $G$ has either a degree-$2$ vertex or a $K_1$- or $K_2$-cut.
\item[(2.2)] Either $a_1=b_1=1$ or $a_2=b_2=1 $.
\end{itemize}
\end{itemize}
\end{lemma}
Let $H$ be a balanced $K_4$-subdivision of a graph $G\in \mathcal{H}_{\ell}$ with $\ell\geq2$. 
If the ear of $H$ shared by its even face cycles has length one, we say that $H$ is of {\em type-$1$}; otherwise we say it is of {\em type-$\ell$}. 
When $\ell>3$, by Lemma \ref{bal K4}, if $G$ has neither a degree-$2$ vertex nor a $K_1$- or $K_2$-cut, then either the ear of a balanced $K_4$-subdivision of $G$ shared by its even face holes has length one or two adjacent ears shared by its face holes of different parity have length one. This obvious corollary of Lemma \ref{bal K4} (2) will be frequently used in the proofs of Lemmas \ref{prism} and \ref{parallel jumps}.

\begin{proof}[Proof of Lemma \ref{bal K4}.]
First we show that (1) holds. One one hand, since $C_1, C_2, C_1\Delta C_3$, and $C_2\Delta C_3$ are odd cycles, $$|C_1|+| C_2|+| C_1\Delta C_3|+|C_2\Delta C_3|=2(a+b+2q+d)\geq4(2\ell +1).$$ On the other hand, since  $C_3, C_4$ are even holes, $a+b+2d=4\ell$. Therefore, $2q-d\geq 2$, implying $q>1$ as $d\geq1$. This proves (1).

Next we show that (2) holds. Since $C_1\cup C_2$ is an induced theta-subgraph of $G$ as $d>1$, by applying Lemma \ref{easy case} (2) to $C_1\cup C_2$, we have $q>\ell$ as $q>1$ by (1). By applying Lemma \ref{easy case} (1) to $H\del V(Q^*)$, we have $a=b=d=\ell$. 

To finish the proof of (2), in the rest of the proof, we may assume that $\ell\geq4$. Without loss of generality, we may further assume that $H$ is chosen with $|H|$ minimum over all balanced $K_4$-subdivisions of type-$\ell$.

\begin{claim}\label{1-neighbour}
Each vertex in $V(G)-V(H)-N_G(\{u_3,u_4\})$ has at most one neighbour in $H$.
\end{claim}
\begin{proof}[Subproof.]
Assume to the contrary that some vertex $v\in V(G)-V(H)-N_G(\{u_3,u_4\})$ has at least two neighbours in $H$. Assume that $v$ only has neighbours in $Q$, implying that $v$ has at least two neighbours in $Q$. Let $v_i$ be the neighbour of $v$ in $Q$ closest to $u_i$ along $Q$ for each integer $1\leq i\leq2$. Since $v$ has no neighbour in $(C_3\cup C_4)\del V(Q)$, by Lemma \ref{even K4}, $G[V(H)\cup\{v\}]\del V(Q^*(v_1,v_2))$ is a balanced $K_4$-subdivision of type-$\ell$ with $|H'|<|H|$, a contradiction. So $v$ has a neighbour in $(C_3\cup C_4)\del V(Q)$.
Since $C_3, C_4$ are even holes and $g(G)=2\ell$, the vertex $v$ has a unique neighbour, say $x$, in $C_3\cup C_4$ by Lemma \ref{neighbour in C}, implying that $v$ has a neighbour in $Q^*$. 

We claim that $x\in V(D^*)$. Assume not. By symmetry assume that
$x\in V(A_2)-\{u_1\}$. Let $y$ be the neighbour of $v$ in $Q$ closest to $u_2$. Since $a=\ell$ implies $a_2<\ell$, we have $|Q(u_1,y)|\geq\ell-1\geq3$. Moreover, since $x\notin\{u_3,u_4\}$, by Lemma \ref{even K4}, $C_3\cup C_4\cup xvyQ(y,u_2)$ is a balanced $K_4$-subdivision of type-$\ell$ whose edge number is less than $|H|$, a contradiction. 

Let $y$ be the neighbour of $v$ in $Q$ closest to $u_2$. Since $v$ has at most one neighbour in $C_3\cup C_4$, we have $y\in V(Q^*)$. Since $Q(u_1,y)yvx$ contains a short jump over $C_4$ by Lemma \ref{find short jump}, we have
$yu_1\notin E(G)$ by Lemma \ref{jump-length}. Since $C_3\cup C_4\cup xvyQ(y,u_2)$ is a balanced $K_4$-subdivision of type-$\ell$, by the minimal choice of $H$, we have $|Q(u_1,y)|=2$, implying $|Q(u_2,y)|\geq3$ as $q>\ell\geq4$. Then $C_3\cup C_4\cup xvyQ(y,u_1)$ is a balanced $K_4$-subdivision of type-$\ell$ whose edge number is less than $|H|$, a contradiction.
\end{proof}
Let $s,t$ be the neighbours of $u_3,u_4$ in $D$, respectively. Set $A_3:=N_G(u_3)-V(C_1)$ and $A_4:=N_G(u_4)-V(C_2)$. Note that $s\in A_3$ and $t\in A_4$. Since $d=\ell\geq4$, we have $D^*\del\{s,t\}\neq\emptyset$. 
\begin{claim}\label{cut-3}
$C_1\cup C_2$ and $D^*\del\{s,t\}$ are in different components of  $G\del (A_3\cup A_4)$. 
\end{claim}
\begin{proof}[Subproof.]
Assume not. 
Let $R$ be a shortest $(C_1\cup C_2, D^*\del\{s,t\})$-path in $G\del(A_3\cup A_4)$. Let $u,v$ be the ends of $R$ with $u\in V(C_1\cup C_2)$ and $v\in V(D^*)\del\{s,t\}$. By the definition of $A_3, A_4$, we have that $u\notin \{u_3,u_4\}$ and $\{u_3,u_4\}$ and $R^*$ are anti-complete. By the choice of $R$ and \ref{1-neighbour}, we have that {\bf (a)} $|R|\geq3$, the subgraph $(H\cup R)\del\{s,t\}$ of $G$ is induced, and $s,t$ may have a neighbour in $(R^*)^*$ but are anti-complete to the ends of $R^*$.

Assume that both $s$ and $t$ have neighbours in $(R^*)^*$. Then there are $s',t'\in V((R^*)^*)$ such that $ss',tt'\in E(G)$ and $R^*(s',t')$ is anti-complete to $H$ by (a). Since $ss'R(s',t')t't$ is a short $(s,t)$-jump over $C_3$, by Lemma \ref{jump-length}, it has length at least $\ell+1$ as $b=d=\ell$. So $C_1\cup C_2\cup u_3ss'R(s',t')t'tu_4$ is an odd $K_4$-subdivision. Since $a=b=\ell>2$, it follows from Lemma \ref{odd K4} that (2.2) holds. Hence, by symmetry and (a) we may assume that $s$ is anti-complete to $R^*$.

We claim that (2.2) holds when $u\notin V(Q)$. First consider the case $u\in V(A_2B_2)-\{u_1,u_2\}$, say $u\in V(A_2)-\{u_1,u_4\}$. Set $H':=C_1\cup C_2\cup RD(v,u_3)$. Since $v\neq t$ and $s$ is anti-complete to $R^*$, we have that $H'$ is an induced subgraph isomorphic to a subdivision of $K_4$ by (a). When $H'$ is balanced, it must be of type-$\ell$ as $|R|\geq3$, so following a similar way proving $a=b=d=\ell$ we have $u=u_4$, a contradiction to the fact $u\notin \{u_3,u_4\}$. Hence, $H'$ is odd. Using Lemma \ref{odd K4} again, $a_1=1$ and either $b_1=1$ or $uu_1\in E(G)$ as $b=\ell$ and $|R|\geq3$. When $b_1=1$, (2.2) holds. So $uu_1\in E(G)$, implying $|A_2(u,u_4)|=\ell-2\geq2$ as $a=\ell$. When $t$ is anti-complete to $R^*$, since $H'$ is odd and $C_3$ is even, $A_2(u,u_4)D(u_4,v)R$ is an odd hole, so $uu_1QB_2D(u_4,v)R$ is an even hole of length at least $2\ell+2$ as $q>\ell$, a contradiction. So $t$ has a neighbour in $R^*$. Let $w\in N(t)\cap V(R^*)$ be closest to $u$ along $R$. For a similar season, $A_2(u,u_4)u_4twR(w,u)$ is an even hole, so $u_4twR(w,u)uu_1u_3B_1B_2$ is an even hole of length $2\ell+2$ as $b=\ell$, a contradiction. Secondly consider the case $u\in V(A_1B_1)-\{u_1,u_2\}$. Let $P$ be a shortest $(t,u)$-path with $V(P^*)\subseteq V(D(t,v)\cup R^*)$. Then $C_1\cup C_2\cup u_4tP$ is an induced subgraph isomorphic to a subdivision of $K_4$. Following a similar way to deal with the first case, we can prove (2.2) true.

We may therefore assume that $u\in V(Q)$. When $t$ has no neighbour in $(R^*)^*$, set $P:=R$ and $w:=v$; and when $t$ has a neighbour in $(R^*)^*$, set $w:=t$ and let $P$ be a shortest $(t,u)$-path with $P^*\subseteq R^*$.
Set $C'_3:=B_2Q(u_2,u)PD(w,u_4)$. Since $C_2$ is odd, by symmetry we may assume that $C'_3$ is an even cycle. Assume that $C'_3$ is not induced. Then $u=u_1$ and $a_2=1$, implying $|P|>\ell$ as $A_2PD(w,u_4)$ is an odd hole. Moreover, since $a_1=\ell-1>1$, we have that $C_3\Delta C'_3$ is an even hole of length larger than $2\ell$ as $q>\ell$, a contradiction. Hence, $C'_3$ is an even hole. By applying Lemma \ref{easy case} (1) on $C_3\cup C'_3$, we have 
$|D(w,u_4)B_2|=|Q(u_2,u)P|=|B_1D(u_3,w)|=\ell$, so $|D(w,u_4)|=b_1$ and $|D(u_3,w)|=b_2$ as $b=\ell$.
Using Lemma \ref{easy case} (2) on $C_2\cup PD(w,u_4)$, we have $a_2+|Q(u_1,u)|>\ell$, so $|Q(u_1,u)|>1$ as $a_2<\ell$. Hence, $C_1\Delta C_2\Delta C'_3$ is an even cycle of length larger than $2\ell$ as $a=\ell$ and $|C'_3|=2\ell$, so $|D(u_3,w)|=b_2=1$ as $|D(u_3,w)|=b_2$. That is, $w=s$, which is not possible as $w\in\{t,v\}$ and $v\notin\{s,t\}$. 
\end{proof}

Let $szs'$ be the 3-vertex path on $D^*$. Note that $s'$ maybe $t$.  We may assume that $z$ has degree at least three, for otherwise (2.1) holds. 
Let $Z$ be a connected induced subgraph of $G\del(A_3\cup A_4)$ containing a neighbour of $z$ and disjoint with $D^*$. We may assume that $Z$ is chosen with $|Z|$ as large as possible. 

We claim that $Z$ is anti-complete to $V(H)\del\{s,z,s'\}$. Since $Z$ is in a component of $G\del(A_3\cup A_4)$ not containing $C_1\cup C_2$, the graph $Z$ is anti-complete to $C_1\cup C_2$.
Assume that $Z$ is not anti-complete to $A_4\del\{s'\}$. 
Let $R$ be a shortest $(s,u_4)$-path with $V(R^*)\subseteq A_4\cup Z\cup\{z\}$. Then $C_1\cup C_2\cup Rsu_3$ is an induced subgraph isomorphic to a subdivision of $K_4$. When it is odd, by Lemma \ref{odd K4}, (2.2) holds from the facts that $a=b=\ell$ and $|R|\geq2$; and when it is balanced, following a similar way proving $a=b=d=\ell$ we have $|R|=\ell-1$, so $RD(s,u_4)$ contains a cycle of length at most $2\ell-2$, a contradiction. Hence, $Z$ is anti-complete to $A_4\del\{s'\}$. 
Similarly, we can further prove that $Z$ is anti-complete to $A_3\del\{s\}$ and $D(s',t)\del s'$. This proves the claim.

Assume that $s,s'$ have a neighbour in $Z$. Let $R$ be a shortest $(s,s')$-path with $V(R^*)\subseteq Z$. Since $Rszs'$ is a cycle, $|R|\geq2\ell-2$, so $B_2B_1u_3sRD(s',u_4)$ is an odd hole by the claim proved in the last paragraph. Then $C_1\cup C_2\cup u_3sRD(s',u_4)$ is an odd $K_4$-subdivision, so (2.2) holds from Lemma \ref{odd K4}. So we may assume that $s$ or $s'$ has no neighbour in $Z$, implying that $\{s',z\}$ or $\{s,z\}$ is a cut of $G$, so (2.1) holds.
\end{proof}

\begin{figure}[htbp]
\begin{center}
\includegraphics[height=6cm,page=2]{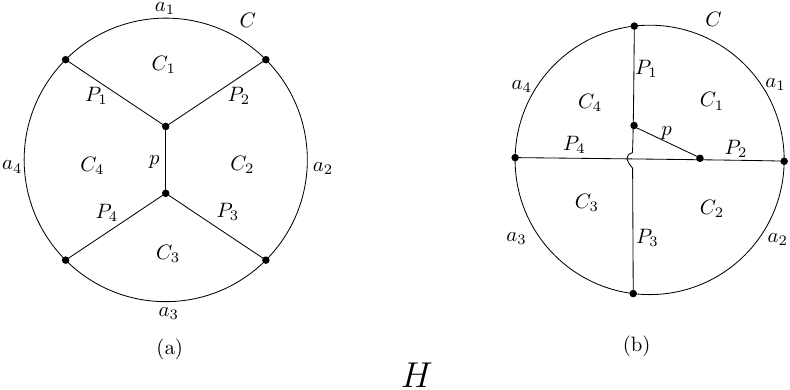}
\caption{
Let $C,C_1,C_2,C_3,C_4$ be the face cycles of $H$. Set $p_i:=|P_i|$ for any integer $1\leq i\leq4$. Let $a_1, a_2, a_3, a_4,p$ denote the lengths of its corresponding ears of $H$. For each $1\leq i\leq4$, we have $p_i,a_i\geq1$. But $p$ maybe equal to $0$.}
\label{w4}
\end{center}
\end{figure}

\begin{lemma}\label{prism}
Let $\ell\geq4$ be an integer and $H$ be an induced subgraph of a graph $G\in\mathcal{H}_{\ell}$. Assume that $H$ is pictured as Figure \ref{w4} and $C$ is an even hole of $G$. If $G$ has neither a degree-$2$ vertex nor a $K_1$- or $K_2$-cut, then one of the following holds. 
\begin{itemize}
\item[(1)] Either $a_1=a_3=1$ or $a_2=a_4=1$.
\item[(2)] $H$ is pictured as Figure \ref{w4} (a), $p\geq2\ell-2$, $\{a_1,a_3\}=\{1,\ell\}$ and $a_i=p_i=p_{i+1}=1$ for some integer $i\in\{2,4\}$, where subscripts are modulo $4$.
\item[(3)] $H$ is pictured as Figure \ref{w4} (b), and $a_i=a_{i+1}=1$ for some integer $1\leq i\leq4$, where subscripts are modulo $4$.
\end{itemize}
\end{lemma}
\begin{proof}
Since $C$ is an even hole, there are exactly two or four holes in $\{C_1,C_2, C_3,C_4\}$ are odd.

Case 1. $p=0$.

Note that the two graphs (a) and (b) in Figure \ref{w4} are the same when $p=0$. 
Since no vertex outside $C$ has two neighbours in $C$ by Lemma \ref{neighbour in C}, by symmetry we may assume that $p_1,p_2,p_3>1$. We claim that $|C_1|,|C_2|$ have different parity. Assume not. Then $C_1\Delta C_2 $ is an even hole. Since $G$ has no even $K_4$-subdivision by Lemma \ref{even K4}, $H\del E(P_2)$ is a balanced $K_4$-subdivision of type-$\ell$. Hence, either $p_1=a_4=1$ or $p_3=a_3=1$ by Lemma \ref{bal K4} (2), a contradiction as $p_1,p_3>1$. Hence, $|C_1|,|C_2|$ have different parity. By symmetry assume that $C_1$ is odd and $C_2$ is even.

Since $H\del E(P_1)$ is not an even $K_4$-subdivision by Lemma \ref{even K4}, $C_4$ is even, so $C_3$ is odd as $C$ is even. Since $H\del E(P_1)$ is a balanced $K_4$-subdivision of type-$\ell$ not satisfying Lemma \ref{bal K4} (2.2) when $a_2>1$, we have $a_2=1$. Since $H\del E(P_3)$ is a balanced $K_4$-subdivision of type-$\ell$ not satisfying Lemma \ref{bal K4} (2.2) when $a_4>1$, we have $a_4=1$. So (1) holds.

Case 2. $p>0$ and $H$ is isomorphic to Figure \ref{w4} (a).

By Lemma \ref{neighbour in C} and symmetry we may assume that $p_1>1$. We claim that (1) holds when $|C_1|,|C_4|$ have the same parity. 
Since $H\del E(P_1)$ is a balanced $K_4$-subdivision of type-$\ell$, it follows from Lemma \ref{bal K4} that $p_3>1$, $p_4=a_3=1$ and $C_2, C_3$ are odd as $p>0$. Then $C_2\Delta C_3$ is an even hole. 
Since $H\del E(P_3)$ is a balanced $K_4$-subdivision of type-$\ell$, by Lemma \ref{bal K4} (2) again, $p_2=a_1=1$, so (1) holds. 

We may therefore assume that $|C_1|,|C_4|$ have different parity, implying that exactly two holes in $\{C_1,C_2, C_3,C_4\}$ are odd. So exactly one of the following holds.
\begin{itemize}
\item[(C1)] $|C_1|$ and $|C_2|$ have the same parity, and $|C_3|,|C_4|$ have the same parity. So $p_2=p_4=1$ by the symmetry of the claim proved in the last paragraph.
\item[(C2)] $|C_1|,|C_3|$ have the same parity, and $|C_2|,|C_4|$ have the same parity.
\end{itemize}
Assume that (C1) holds. Since $p_4=1$ implies $p_3>1$, by symmetry we may assume that $|C_1|,|C_2|$ are odd, and $|C_3|,|C_4|$ are even. Then $H\del E(P_3)$ is a balanced $K_4$-subdivision whose ear $P_2$ shared by its odd holes has length one, a contradiction to Lemma \ref{bal K4} (1). So (C2) holds.

First we consider the case that $|C_1|,|C_3|$ are odd, and $|C_2|,|C_4|$ are even. Since $H\del E(P_1)$ is a balanced $K_4$-subdivision, $p_4>1$ by Lemma \ref{bal K4} (1). When $a_2>1$, since $H\del E(P_1)$ and $H\del E(P_4)$ are balanced $K_4$-subdivisions of type-$\ell$, we have $a_1=a_3=1$ by Lemma \ref{bal K4} (2), so (1) holds. Hence, we may assume that $a_2=1$. Then $a_4>1$, for otherwise (1) holds. Since $H\del E(P_3)$  is a balanced $K_4$-subdivision of type-$\ell$ not satisfying Lemma \ref{bal K4} (2.2) when $p_3>1$, we have $p_3=1$. Similarly, $p_2=1$. Hence, $p=2\ell-3$ as $C_2$ is an even hole. Then $C_4$ is an even hole of length larger than $2\ell$ as $a_4,p_4,p_1>1$, a contradiction.

Secondly, we consider the case that $|C_1|,|C_3|$ are even, and $|C_2|,|C_4|$ are odd. Since $p>0$ implies that $C\cup C_i$ is induced for each $i\in\{1,3\}$, by Lemma \ref{easy case} (1), either $a_i=\ell$ or $a_i=1$. By symmetry assume that $a_1=\ell$ and $a_3=1$ as $|C|=2\ell$. Since $H\del E(P_3)$ is a balanced $K_4$-subdivision of type-$\ell$ not satisfying Lemma \ref{bal K4} (2.2) when $p_3>1$, we have $p_3=1$, implying $p_4>1$. Since $H\del E(P_4)$ is a balanced $K_4$-subdivision of type-$\ell$, by Lemma \ref{bal K4} (2.2) again, we have $a_2=p_2=1$, so $p\geq2\ell-2$ as $C_2$ is odd. That is, (2) holds.

Case 3. $p>0$ and $H$ is isomorphic to Figure \ref{w4} (b).

Let $P$ be the ear of $H$ with length $p$. When $p>1$, since $H\del E(P)$ is an odd $K_4$-subdivision or a balanced $K_4$-subdivision of type-$\ell$, (3) holds by Lemma \ref{odd K4} or Lemma \ref{bal K4} (2.2). So we may assume that $p=1$.
By Lemma \ref{neighbour in C} and symmetry we may assume that $p_1>1$. We claim that $|C_1|,|C_4|$ have different parity. Assume not. Since $H\del E(P_1)$ is a balanced $K_4$-subdivision of type-$\ell$, it follows from Lemma \ref{bal K4} that $p_3\geq\ell$ (as $p=1$), both $C_2$ and $C_3$ are odd, and either $p_4=a_3=1$ or $p_2=a_2=1$. By symmetry assume $p_2=a_2=1$, implying $p_4>1$. Since $H\del E(P_3)$ is a balanced $K_4$-subdivision of type-$\ell$, by Lemma \ref{bal K4} (2) again, $a_1=1$ and $C_1, C_4$ are odd. 
Moreover, since $C_2, C_3$ are also odd and $C$ is even, $p_1+p_3=p_2+p_4=\ell$ by Lemma \ref{easy case} (1), so
\[4(2\ell+1)\leq|C_1|+|C_2|+ |C_3|+ |C_4|=|C|+2(p_1+p_2+p_3+p_4)+4p=6\ell+4,\] a contradiction.

We may therefore assume that $C_1$ is even and $C_2$ is odd, so there are exactly two holes in $\{C_1,C_2, C_3,C_4\}$ are odd as $C$ is even. Moreover, since $p_2>1$ or $p_4>1$, by symmetry and the claim proved in the last paragraph, $C_4$ is even and $C_3$ is odd. Note that $H\del E(P_1)$ is a balanced $K_4$-subdivision. So $p_2>1$ by Lemma \ref{bal K4} (1). When $a_3>1$, since $H\del E(P_1)$ is  of type-$\ell$ not satisfying Lemma \ref{bal K4} (2.2), we have $a_3=1$, implying $a_1>1$ for otherwise (1) holds. Since $H\del E(P_4)$ is a balanced $K_4$-subdivision of type-$\ell$ not satisfying Lemma \ref{bal K4} (2.2) when $p_4>1$, we have $p_4=1$. Similarly, $p_3=1$, implying $|C_3|=4$ as $p=1=a_3$, a contradiction.
\end{proof}

\section{Proof of Theorem \ref{main thm+}}
Let $C$ be a cycle of a graph $G$ and $e,f$ be chords of $C$. If no cycle in $C\cup e$ containing $e$ contains the two ends of $f$, we say $e,f$ are {\em crossing}; otherwise they are {\em parallel}. Note that by our definition, $e,f$ are parallel when they share an end. When $C$ is a hole and $P_i$ is an $(u_i,v_i)$-jump over $C$ for each integer $1\leq i\leq2$, if $u_1v_1$ and $u_2v_2$ are parallel of $C$ after we replace $P_1$ by the edge $u_1v_1$ and $P_2$ by the edge $u_2v_2$, we say that $P_1,P_2$ are {\em parallel}; otherwise, they are {\em crossing}.

\begin{lemma}\label{type-e}
Let $\ell\geq3$ be an integer and  $C$ be an even hole of a graph $G\in \mathcal{H}_{\ell}$. Then all short jumps over $C$ of type-e have the same ends.
\end{lemma}
\begin{proof}
Assume to the contrary that $P_1$ and $P_2$ are short jumps over $C$ of type-e whose ends are not the same. Then $P_1$ and $P_2$ are crossing and $|P_1|=|P_2|=\ell$ by Lemma \ref{jump-length}. Since $|C|=2\ell=g(G)$, we have $V(P_1)\cap V(P_2)=\emptyset$, so there is a cycle $C'$ in $C\cup P_1\cup P_2$ containing $P_1\cup P_2$ with $2\ell+2\leq|C'|\leq3\ell$. 
Since the two holes in $C\cup P_i$ containing $P_i$ are even for each integer $1\leq i\leq 2$, we have that $|C'|$ is even, so either $C'$ is an even hole of length at least $2\ell+2$ or $G[V(C')]$ contains a hole of length less than $2\ell$ as $\ell\geq3$, which is not possible as $g(G)=2\ell$.
\end{proof}

Let $s,t\in V(C)$ be nonadjacent. 
Except there are short $(s,t)$-jumps of type-o and type-e at the same time, all $(s,t)$-jumps have the same parity although their lengths maybe different. In this paper, we only put attention on the parity of jumps over $C$ and do not care which jump is chosen. That is, if we say $P_1,P_2$ are jumps over $C$, we mean that they do not have the same ends.
\begin{lemma}\label{jump-cycle-parity}
Let $\ell\geq2$ be an integer and  $C$ be an even hole of a graph in $\mathcal{H}_{\ell}$. For any $1\leq i\leq2$, let $P_i$ be a short or local jump over $C$ such that it is of type-o when it is short and such that it is across one vertex when it is local. Assume that $P_1$ and $P_2$ are parallel and internally  disjoint. Then $P_1^*$ and $P^*_2$ are not anti-complete. 
\end{lemma}
\begin{proof}
For any $1\leq i\leq2$, since $|P_i|>2\ell$ when it is local, the two cycles in $C\cup P_i$ containing $P_i$ are odd no matter whether $P$ is short or local. Moreover, since $C$ is even, so is the cycle in $C\cup P_1\cup P_2$ containing $P_1\cup P_2$, which is not induced since it has length at least $2\ell+2$ by Lemma \ref{jump-length} possibly. Hence,  $P_1^*$ and $P^*_2$ are not anti-complete.
\end{proof}

\begin{lemma}\label{4-vtx}
Let $\ell\geq4$ be an integer and  $C$ be an even hole of a graph $G\in \mathcal{H}_{\ell}$. For any $1\leq i\leq2$, let $P_i$ be a short $(u_i,v_i)$-jump over $C$ 
such that $P_1$ and $P_2$ are of type-o when they are parallel. If $G$ has neither a degree-$2$ vertex nor a $K_1$- or $K_2$-cut, then one of the following holds.
\begin{itemize}
\item[(1)] When $u_1=u_2$, we have $v_1v_2\in E(G)$. 
\item[(2)] When $|\{u_1,u_2,v_1,v_2\}|=4$, either $u_1u_2,v_1v_2\in E(G)$ or $u_iu_{3-i}v_i$ is a path on $C$ for some integer $1\leq i\leq2$. 
\end{itemize}
\end{lemma}
\begin{proof}
Set $S:=\{u_1,u_2,v_1,v_2\}$. 
Let $P_1$ and $P_2$ be chosen with $|P_1\cup P_2|$ as small as possible. Then $P_1\cup P_2$ is a forest whose leaf vertices are contained in $S$. Since $P_1$ and $P_2$ are short jumps over $C$, we have that {\bf (a)} $(P_1^*)^*\cup (P_2^*)^*$ and $C$ are anti-complete.

First we consider (1). 
By symmetry we may assume that $|C(u_1,v_2)|\neq\ell$. 
Assume that there is an induced tree $T$ of $P_1\cup P_2$ with $S$ as its set of leaf vertices. Since no vertex outside $C$ has two neighbours in $C$ by Lemma \ref{neighbour in C}, $C\cup T$ is a balanced $K_4$-subdivision by (a) and Lemma \ref{even K4}. When $C\cup T$ is of type-1, since $|C(u_1,v_1)|, |C(u_1,v_2)|>1$, the ear shared by even holes of $C\cup T$ must be $C(v_1,v_2)$, so $v_1v_2\in E(G)$. When $C\cup T$ is of type-$\ell$, since $|C(u_1,v_1)|, |C(u_1,v_2)|>1$, the hole of $C\cup T$ containing $C(v_1,v_2)$ not $C$ must be odd, for otherwise we get a contradiction to Lemma \ref{bal K4} (2.2). Using Lemma \ref{bal K4} (2.2) again, $v_1v_2\in E(G)$.
Hence, it suffices to show that such $T$ exists.

When $P_1\cup P_2$ is not a path, that is, it is  isomorphic to a subdivision of $K_{1,3}$, since $S$ is the set of degree-1 vertices of $P_1\cup P_2$, such $T$ exists by Lemma \ref{induced tree}. Hence, we may assume that $P_1\cup P_2$ is a $(v_1,v_2)$-path. Since $P_1$ and $P_2$ are short jumps of type-o, by Lemma \ref{jump-cycle-parity}, 
there is an edge $uv\in E(G)-E(P_1\cup P_2)$ with $u\in V(P_1^*)$ and $v\in V(P_2^*)$. 
Since $g(G)=2\ell$, either $uu_1\notin E(G)$ or $vu_1\notin E(G)$, say $uu_1\notin E(G)$. 
Since $G[V(P_2\cup P_1(u,v_1))]\del E(C)$ is connected and  has $S$ as its set of degree-1 vertices, it contains such an induced tree $T$ by Lemma \ref{induced tree}, so (1) holds.

Next we consider (2). 
Set $H:=G[V(P_1\cup P_2)]\del E(C)$.
We claim that (2) holds when $H$ is connected. By (a), 
$S$ is the set of degree-1 vertices of $H$. Hence, by Lemma \ref{induced tree}, there is an induced tree $T$ of $H$ containing $S$. 
By (a) and Lemma \ref{neighbour in C}, we have that $C\cup T$ is isomorphic to a graph pictured as Figure \ref{w4}. So (2)  holds from Lemma \ref{prism}.

We may therefore assume that $H$ is not connected.
Since $H$ is connected when $V(P_1^*)\cap V(P_2^*)\neq\emptyset$ or there is some edge linking $P_1^*$ and $P_2^*$, the path $P_1^*$ and $P_2^*$ are anti-complete. When  $P_1$ and $P_2$ are parallel, since they are of type-o, $H$ is connected by Lemma \ref{jump-cycle-parity}, so
$P_1$ and $P_2$ are crossing, implying that $C\cup P_1\cup P_2$ is induced and isomorphic to a subdivision of $K_4$. Since $G$ has no even $K_4$-subdivision by Lemma \ref{even K4}, it is odd or balanced. When $C\cup P_1\cup P_2$ is odd, it follows from Lemma \ref{odd K4} that $u_iu_{3-i}v_i$ is a path on $C$ for some integer $1\leq i\leq2$ as $|P_1|,|P_2|>1$. 
When $C\cup P_1\cup P_2$ is balanced, implying that exactly one of $P_1$ and $P_2$ is of type-o, it follows from Lemma \ref{bal K4} (2.2) that (2) holds.
\end{proof}
\begin{lemma}\label{parallel jumps}
Let $\ell\geq4$ be an integer and $C$ be an even hole of a graph $G\in \mathcal{H}_{\ell}$. Let $P_1$ be a local $(u_1,v_1)$-jump over $C$  across one vertex $s_1$. Let $P_2$ be a short or local $(u_2,v_2)$-jump over $C$. 
Assume that $P_1,P_2$ are parallel and 
$G$ has neither a degree-$2$ vertex nor a $K_1$- or $K_2$-cut. Then the following hold.
\begin{itemize}
\item[(1)] When $P_2$ is local and across one vertex $s_2$, if $P_1,P_2$ do not share an end, then there is a short $(s_i,w_j)$-jump over $C$, where $\{i,j\}=\{1,2\}$ and $w_j\in\{u_j, v_j,s_j\}$.
\item[(2)] When $P_2$ is short and of type-o, $P_1$ and $P_2$ share an end. 
\end{itemize}
\end{lemma}
\begin{proof}
Assume that $u_1,u_2,v_2,v_1$ appear on $C$ in this order.
Set $S:=\{u_1,u_2,v_1,v_2\}$. By the aim of the lemma, we may assume that $|S|=4$.
Let $P_1,P_2$ be chosen with $|P_1\cup P_2|$ as small as possible. Then $P_1\cup P_2$ is a forest with $S$ as its set of degree-1 vertices.

First we consider the case that $P_1\cup P_2$ is connected. Note that $P_1\cap P_2$ and $C$ are anti-complete by the definitions of local and short jumps. When $P_2$ is local, since there is an $(s_1,s_2)$-path $P$ with $P^*\subseteq (P_1^*)^*\cup(P_2^*)^*$, there is a short $(s_1,s_2)$-jump $R$ with $V(R)\subseteq V(P)$ by Lemma \ref{find short jump}, so (1) holds. Hence, we may assume that $P_2$ is short and of type-o. Following a similar way as above, there are short  $(s_1,u_2)$- and  $(s_1,v_2)$-jumps. By Lemma \ref{type-e} and symmetry we may assume that there is a short $(s_1,u_2)$-jump over $C$ of type-o. Since $P_2$ is also a short jump of type-o, by Lemma \ref{4-vtx} (1), $s_1v_2\in E(G)$, which is not possible as $v_1\neq v_2$. 

Secondly, we consider the case that $P_1\cup P_2$ is not connected. 
By Lemma \ref{jump-cycle-parity}, there is an edge $xy\notin E(P_1\cup P_2)$ with $x\in V(P_1^*)$ and $y\in V(P_2^*)$. If possible, $x,y$ are chosen with $|\{x,y\}\cap N(C)|$ as small as possible. When $|\{x,y\}\cap N(C)|=0$, following a similar way as the last paragraph, we can prove the lemma. So we may assume that $|\{x,y\}\cap N(C)|\geq1$, implying that $(P_1^*)^*$ and $(P_2^*)^*$ are anti-complete by the choice of $xy$. Since $g(G)=2\ell\geq8$, we have $|\{x,y\}\cap N(C)|=1$ by Lemma \ref{jump-length}.

Assume that $x\notin N(C)$, implying $y\in N(C)$. By symmetry assume $yv_2\in E(G)$. By Lemma \ref{find short jump}, there is a short $(s_1,v_2)$-jump $P$ over $C$. So (1) holds when $P_2$ is local. Hence, we may assume that $P_2$ is short.
When  $P$ is of type-o, using Lemma \ref{4-vtx} (1) to $P_2,P$, we get a contradiction again.
Hence, $P$ is of type-e. 
Then $|P|=|C(s_1,v_1,v_2)|=\ell$ by Lemma \ref{jump-length}. Let $y'\in V(P_2^*)$ have a neighbour in $P\cap P_1$ closest to $u_2$ along $P_2$. Note that $y'$ maybe equal to $y$ but can not be adjacent to $y$ as $g(G)=2\ell$. Since $|P|=\ell$, we have $|N(y')\cap V(P\cap P_1)|=1$. Let $x'$ be the vertex in $N(y')\cap V(P\cap P_1)$. 
Then $C\cup P\cup x'y'P_2(y',u_2)$ is a balanced $K_4$-subdivision of type-$\ell$ by
Lemma \ref{even K4}, which does not satisfy Lemma \ref{bal K4} (2.2) as $|C(u_2,v_2)|>1$ and $u_1\neq u_2$, a contradiction.

We may therefore assume that $x\in N(C)$, say $xu_1\in E(G)$, implying $y\notin N(C)$. 
By Lemma \ref{neighbour in C}, symmetry and the case dealt with in the last paragraph, we may further assume that $(P_1^*)^*$ and $P_2$ are anti-complete, and $P_2$ is short. Let $y_1,y_2\in N(x)\cap V(P_2)$ be closest to $u_2,v_2$, respectively. Note that $y_1$ maybe equal to $y_2$. By Lemma \ref{jump-length} and the fact $g(G)=2\ell$, we have $y_1,y_2\notin N(C)$. Let $x_1$ be the neighbour of $s_1$ closest to $x$ along $P_1$. Since $(P_1^*)^*$ and $P_2$ are anti-complete, $C\cup P_1(u_1,x_1)x_1s_1\cup xy_1P_2(y_1,u_2)\cup xy_2P_2(y_2,v_2)$ is pictured as Figure \ref{w4} (a) with $p\in\{0,1\}$, so we get a contradiction to Lemma \ref{prism}.
\end{proof}

\begin{lemma}\label{s-path}
Let $\ell\geq4$ be an integer and  $C$ be an even hole of a graph in $\mathcal{H}_{\ell}$. Let $S$ be a set of ends of all short jumps over $C$. Assume that $G$ has neither a degree-$2$ vertex nor a $K_1$- or $K_2$-cut and some short jump over $C$ is of type-o. 
Then $V(Q^*)\cap S\neq\emptyset$ for any $6$-vertex path $Q$ on $C$.
\end{lemma}
\begin{proof}
Set $Q:=s_1w_1w_2w_3w_4s_2$. Assume to the contrary that $V(Q^*)\cap S=\emptyset$. Since some short jump over $C$ is of type-o, by Lemma \ref{parallel jumps} (2), no local jump over $C$ across one vertex has its ends in $Q^*$. Hence, by Lemma \ref{xyz}, 
$w_2$ is an end of a local jump $P$ across one vertex. 
By Lemma \ref{parallel jumps} (2), the other end of $P$ must be $s_1$ and $s_1\in S$. Similarly, there is a local $(s_2,w_3)$-jump $P'$ across $w_4$. By applying Lemma \ref{parallel jumps} (1) to $P,P'$, either $w_1$ or $w_4$ is in a short jump  over $C$, a contradiction as $V(Q^*)\cap S=\emptyset$.
\end{proof}

\begin{proof}[Proof of Theorem \ref{main thm+}.]
Assume not. Let $C$ be an even hole of a graph $G\in \mathcal{H}_{\ell}$ with $\ell>4$. Let $S$ be the set consisting of ends of short jumps over $C$. Since each vertex not in $C$ has at most one neighbour in $C$ by Lemma \ref{neighbour in C}, to get a contradiction, we only need to find a 3-vertex path $xyz$ on $C$ not satisfying Lemma \ref{xyz}. That is, our aim is to find a 3-vertex path $xyz$ on $C$ satisfying the following statements (a)-(c). {\bf (a)} $\{x,y,z\}\cap S=\emptyset$, {\bf (b)} there is no local $(x,z)$-jump across $y$, and {\bf (c)} $y$ is not an end of a local jump over $C$ across  one vertex. 

We claim that there is a 3-vertex path $xyz$ of $C$ satisfying (a)-(c) when $|S|\leq2$.  
When $S=\emptyset$, that is, there is no short jump over $C$, by Lemma \ref{parallel jumps} (1), every pair of local jumps over $C$ across one vertex either share an end or are crossing, so there is a 5-vertex path $Q$ on $C$ containing the ends of all local jumps over $C$ across one vertex. Then each 3-vertex path $xyz$ of $C\del V(Q)$ satisfies (a)-(c). So we may assume that $|S|=2$. 
Set $S:=\{s,t\}$. Since $|C|\geq10$, at least one $(s,t)$-path on $C$ has length at least 5, so all short $(s,t)$-jumps are of type-e by Lemma \ref{s-path}. 
Since $d_C(s,t)=\ell$ by Lemma \ref{jump-length}, there is a 3-vertex path $xyz$ of $C\del S$ satisfying (a)-(c) by Lemma \ref{parallel jumps}.
We may therefore assume that $|S|\geq3$.
By Lemma \ref{type-e}, there is a short jump over $C$ of type-o. Moreover, by Lemma \ref{parallel jumps} (2), no component of $C\del S$ contains the ends of a local jump over $C$ across one vertex. Hence, to find a 3-vertex path $xyz$ on $C$ satisfying (a)-(c), it suffices to {\bf (d)} find a degree-2 vertex $y$ of $C\del S$ that is not an end of a local jump over $C$ across one vertex, as the fact that $y$ has degree-2 in $C\del S$ implies the neighbours of $y$ on $C$ not in $S$.

For any $1\leq i\leq2$, let $P_i$ be a short $(u_i,v_i)$-jump over $C$. By Lemma \ref{type-e}, we may assume that $P_1$ is of type-o. Set $S':=\{u_1,u_2,v_1,v_2\}$. We may assume that $P_1,P_2$ are chosen with $|S'|$ as large as possible. Subject to this condition, we may further assume that $P_2$ is of type-o if possible. 

\begin{claim}\label{2-comp}
$C\del S'$ has exactly two components.
\end{claim}
\begin{proof}[Subproof.]
Assume that $C\del S'$ has at least three components. Then $P_1, P_2$ are parallel and $P_2$ is of type-e by Lemma \ref{4-vtx}. Assume that $u_1,u_2,v_2,v_1$ appear in this order on $C$. Since $d_C(u_2,v_2)=\ell$ by Lemma \ref{jump-length}, there is a vertex $s\in S\cap V(C^*(u_2,v_2))$ by Lemma \ref{s-path}. Let $P$ be a short $(s,t)$-jump.
Since $P_2$ is of type-e, by Lemma \ref{type-e},  $P$ is of type-o.
Moreover, by the choice of $P_1, P_2$, we have $t\in\{u_1,v_1\}$. When $t=v_1$, by Lemma \ref{4-vtx} (1), we have $su_1\in E(G)$, implying $u_1=u_2$, which is a  contradiction to the choice of $P_1,P_2$.  In a similar way we can show $t\neq u_1$. Hence, it suffices to show that $C\del S'$ is not connected. 

Assume that $C\del S'$ is connected. Then $P_1, P_2$ are crossing. By Lemma \ref{4-vtx} (2) and symmetry, assume that $u_1 u_2v_1v_2$ is a path on $C$. Then $P_2$ is of type-o by Lemma \ref{jump-length}. By Lemma \ref{s-path}, there is a vertex $s\in S-S'$ in $C^*(u_1,v_2)$. Let $P$ be a short $(s,t)$-jump. If possible, $P$ is chosen of type-o. Assume that $P$ is of type-e. Then $S=S'\cup\{s,t\}$ by the choice of $P$. Since $d_C(s,t)=\ell\geq5$, by Lemma \ref{s-path}, $t\in\{u_2,v_1\}$, say $t=v_1$. Let $Q$ be the $(s,t)$-path on $C$ containing $v_2$. 
Let $y\in V(Q)$ with $d_C(y,v_1)=3$. Then $y$ is not in a local jump across one vertex by Lemma \ref{parallel jumps} (2), so $y$ satisfies (d) as $|Q|=\ell$ and $S=S'\cup\{s,t\}$. So $P$ is of type-o.

When $P$ is parallel with $P_1, P_2$, we have $\{s,t\}=\{u_1,v_2\}$ by Lemma \ref{4-vtx}, a contradiction as $s\notin S'$. So by symmetry we may assume that $P,P_2$ are crossing, implying $t=v_1$. Then $su_1\in E(G)$ by Lemma \ref{4-vtx} (1). That is, $su_1 u_2v_1v_2$ is a path on $C$. 
We claim that all short jumps over $C$ of type-o have ends in $S'\cup\{s\}$. Assume to the contrary that there is a short $(u,v)$-jump $P'$ of type-o with $u\notin S'\cup\{s\}$. Since $v_1$ is an end shared by $P,P_1$, 
we have $v_1\neq v$ by Lemma \ref{4-vtx} (1), so $P',P_2$ are parallel. By Lemma \ref{4-vtx}, $uv_2\in E(G)$ and $v\in\{u_1, u_2\}$ as $u\notin S'\cup\{s\}$. Hence, $P,P'$ do not satisfy Lemma \ref{4-vtx} (2), a contradiction.

Since $|C(s,v_2)|\geq5$, by Lemma \ref{s-path}, $S\cap V(C^*(s,v_2))\neq\emptyset$, so there is a short $(s',t')$-jump $P'$ with $s'\notin S'\cup\{s\}$. 
Then $P'$ is of type-e by the claim proved in the last paragraph. Moreover, by Lemma \ref{type-e}, $S=S'\cup\{s,s',t'\}$. Since $d_C(s',t')=\ell$, by Lemma \ref{s-path}, $t'\in\{u_1, u_2,v_1\}$.
Since $d_C(s,v_1)=3\neq\ell$, the jumps $P,P'$ is not crossing by Lemma \ref{4-vtx} (2), so $t'=v_1$. Let $y\in V(C^*(s',v_2))$ satisfy $d_C(y,v_2)=2$. Then $y$ is not an end of a local jump across one vertex by Lemma \ref{parallel jumps} (2), so $y$ satisfies (d) as $S=S'\cup\{s,s'\}$ and $d_C(s',v_1)=\ell$. 
\end{proof}

Since $3\leq|S'|\leq4$, by \ref{2-comp} and symmetry, we may assume that exactly one of the following holds.
\begin{itemize}
\item[(1)] $u_1=u_2$ 
and $v_1v_2\in E(G)$, or
\item[(2)] $u_1u_2$ and $v_1v_2$ are anti-complete edges on $C$, or
\item[(3)] $v_2$ is anti-complete to the 3-vertex path $u_1u_2v_1$ on $C$.
\end{itemize}

\begin{claim}\label{S'=4}
(1) can not happens, so $|S'|=4$.
\end{claim}
\begin{proof}[Subproof.]
Assume not. 
Since $d_C(u_1,v_2)=\ell$ when $P_2$ is of type-e, by symmetry assume that $|C(u_1,v_1)|<|C(u_1,v_2)|$. Then $|C(u_1,v_2)|\geq5$. By Lemma \ref{s-path}, $S\cap C^*(u_1,v_2)\neq\emptyset$. Let $P$ be a short $(s,t)$-jump with $s\in V(C^*(u_1,v_2))$. By the choice of $P_1,P_2$, the jump $P_2$ is of type-o and $t=u_1$, for otherwise we should choose $P,P_2$, not $P_1,P_2$. 
By Lemma \ref{4-vtx} (1), $P$ is of type-e. 
Hence, by the choice of $P_1,P_2$ and Lemma \ref{type-e}, we have $S=S'\cup\{s\}$. However, since $|C(s,u_1)|=\ell\geq5$, by Lemma \ref{s-path} again, $S\cap V(C^*(s,u_1))\neq\emptyset$, a contradiction to the facts $S=S'\cup\{s\}$ and $s\in V(C^*(u_1,v_2))$.
\end{proof}

We claim that either $S=S'$ or $P_2$ is of type-o. Assume not. Since $S'\subsetneqq S$, there is a short $(s,t)$-jump $P$ with $s\notin S'$. Since $P_2$ is of type-e, by Lemma \ref{type-e}, $P$ is of type-o, so by the choice of $P_1,P_2$, we have $t\in\{u_1,v_1\}$, say $t=u_1$. Since $P_1$ is also of type-o, $v_1s\in E(G)$ by Lemma \ref{4-vtx} (1). Since every pair of short jumps over $C$ of type-o must share an end by the choice of $P_1,P_2$, we have $S=S'\cup\{s\}$ by Lemma \ref{type-e}. Moreover, since $d_C(u_2,v_2)=\ell$ by Lemma \ref{jump-length}, it follows from Lemma \ref{s-path} that $P_2$ can not be parallel with $P$ and $P_1$. Since $\ell>3$, $t=u_1$ and $v_1s\in E(G)$, by Lemma \ref{4-vtx} (2), $P_2$ also can not be crossing with $P$ and $P_1$, so $P_2$ is parallel with exactly one of $\{P,P_1\}$. Since $|S'|=4$ by \ref{S'=4}, by Lemma \ref{4-vtx} again, $P_2$ is parallel with  $P$ and $s\in \{u_2,v_2\}$, a contradiction to the fact $s\notin S'$.

We claim that $S\neq S'$, so $P_2$ is of type-o by the claim proved in the last paragraph. Assume to the contrary that $S=S'$. Since $|C|\geq10$, by Lemma \ref{s-path} and \ref{2-comp}, $|Q_1|=|Q_2|=2$, where $Q_1,Q_2$ are the components of $C\del S'$. Let $y_i$ be the unique degree-2 vertex of $Q_i$ for  each integer $1\leq i\leq 2$. Then $y_i$ is an end of a local jump $R_i$ across one vertex, for otherwise $y_i$ satisfies (d). Note that when (3) happens, 
the other end of $R_i$ is $u_1$ or $v_1$ by Lemma \ref{parallel jumps} (2). By applying Lemma \ref{parallel jumps} (1) to $R_1,R_2$, some end of $Q_1$ or $Q_2$ is an end of a short jump, so it should be in $S$, a contradiction to the fact $S=S'$. 


Assume that (2) happens. We claim that all short jumps of type-o have their ends in $S'$. Assume not. Let $P$ be a short $(s,t)$-jump of type-o with $s\notin S'$. When $t\in S'$, say $t=u_1$, we have $sv_1\in E(G)$ by Lemma \ref{4-vtx} (1), so $P,P_2$ do not satisfy Lemma \ref{4-vtx} (2) as $s\notin S'$. Hence, $t\notin S'$, so $P$ is parallel or crossing with $\{P_1, P_2\}$. By Lemma \ref{4-vtx} (2), $P$ must be crossing with $\{P_1, P_2\}$, so $|C|=6$ by Lemma \ref{4-vtx} (2) again, a contradiction. This proves the claim.

Since $S\neq S'$, there is a short $(s,t)$-jump $P$ with $s\notin S'$. Then $P$ is of type-e by the claim proved in the last paragraph, 
so $S=S'\cup\{s,t\}$ by Lemma \ref{type-e}.
Since $|C|\neq6$, by Lemma \ref{4-vtx} (2), $P$ is crossing at most one of $\{P_1,P_2\}$. Moreover, since $d_C(s,t)=\ell\geq5$ and $S=S'\cup\{s,t\}$, by Lemma \ref{s-path}, $P$ is crossing exactly one of $\{P_1,P_2\}$. By symmetry assume that $t=u_1$ and $P,P_2$ are crossing. By applying Lemma \ref{4-vtx} (2) first to $P,P_2$ and then to $P,P_1$, the jumps $P_1,P_2$ are parallel. Let $y\in V(C^*(u_2,v_2))$ with $d_C(y,u_2)=2$. Then $y$ is not in a local jump across one vertex by Lemma \ref{parallel jumps} (2). Moreover, since $d_C(s,u_1)=\ell$, the vertex $y$ is a degree-2 vertex of $C\del S$, so $y$ satisfies (d). So we may assume that (3) happens. 

\begin{claim}\label{location}
Let $R$ be a short $(u,w)$-jump $R$ of type-o with $u\notin S'$. Then either there is a vertex satisfying (d) or $w=u_2$ and $uv_2\in E(G)$.
\end{claim}
\begin{proof}[Subproof.]
Since $u\notin S'$ and $R,P_2$ are of type-o, we have $w\neq v_2$ by Lemma \ref{4-vtx} (1). When $uv_2w$ is a path on $C$, by applying Lemma \ref{4-vtx} (2) to $P_1,R$, we have $|C|=6$, a contradiction. Hence, $uv_2w$ is not a path on $C$. When $w\neq u_2$, since $w\neq v_2$ and $u\notin S'$, replacing $P_1,P_2$ by $R,P_2$ we are back to a case with the same structure as (2), so by symmetry we can find a vertex $y$ satisfying (d). So we may assume that $w=u_2$. By applying Lemma \ref{4-vtx} (1) to $P_2,R$, we have $uv_2\in E(G)$.
\end{proof}

Let $P$ be a short $(s,t)$-jump with $s\notin S'$.
We may assume that $P$ is chosen of type-e if possible. Assume that $P$ is of type-o. By \ref{location}, we have $t=u_2$ and $sv_2\in E(G)$. Since the other end of all short jumps of type-o that have $u_2$ as an end must be in $\{v_2,s\}$ by Lemma \ref{4-vtx} (1), all short jumps over $C$ have their ends in $S'\cup \{s\}$ by \ref{location}. Moreover, by the choice of $P$, we have 
$S=S'\cup \{s\}$. Since $|C|\geq10$, some component $Q$ of $C\del S$ has at least three vertices. Let $y\in V(Q)$ be chosen with $d_C(y,S)$ as large as possible, implying $d_C(y,S)\geq2$. Then $y$ is not an end of a local jump across one vertex by Lemma \ref{4-vtx} (2), so $y$ satisfies (d). Hence, $P$ is of type-e, implying $S\subseteq S'\cup \{v,s,t\}$ for some vertex $v\in V(C)\cap N(v_2)$ by Lemma \ref{type-e} and \ref{location}. 

First we consider the case that $P,P_2$ are crossing. Since $v_2$ anti-complete to $u_1u_2v_1$ implies $d_C(u_2,v_2)\geq3$, by Lemma \ref{4-vtx} (2), there are edges $e,f\in E(C)$ with $V(\{e,f\})=\{s,t,u_2,v_2\}$. 
By symmetry we may assume $t=u_1$, implying $sv_2\in E(G)$ and $s\in V(C^*(v_1,v_2))$. Let $s'$ be the vertex in $V(C)\cap N(v_2)$ not equal to $s$. When $s'\in S$, by \ref{location}, there is a short $(s',u_2)$-jump $R$ of type-o, which is not possible as $R,P$ do not satisfy Lemma \ref{4-vtx} (2). So $s'\notin S$, implying $S\subseteq S'\cup \{s\}$. Let $y\in V(C(u_1,s',v_2))$ with $d_C(y,u_1)=2$. Since $d_C(u_1,s)=\ell$ by Lemma \ref{jump-length}, we have $y\neq s'$. By Lemma \ref{parallel jumps} (2), $y$ is not in a local jump across one vertex. So $y$ satisfy (d) as $S\subseteq S'\cup \{s\}$.
Secondly, we consider the case that $P,P_2$ are parallel. By symmetry assume that $s\in V(C^*(v_1,v_2))$. Since $d_C(s,t)=\ell$ and $S\subseteq S'\cup \{v,s,t\}$, by Lemma \ref{s-path}, $t\in\{u_2,v_2\}$ and when $t=v_2$, the vertex in $N(v_2)\cap V(C^*(v_1,v_2))$ is in $S$. When $t=v_2$, choose $y\in V(C(v_1,v_2))$ with $d_C(y,v_2)=3$; and when $t=u_2$, choose $y\in V(C(v_1,v_2))$ with $d_C(y,v_1)=2$. Since $s\in V(C^*(v_1,v_2))$ and $d_C(s,t)=\ell$, by Lemma \ref{parallel jumps} (2), $y$ is not in a local jump across one vertex,  so $y$ satisfies (d).
\end{proof}

\section{Acknowledgments}
This research was partially supported by grants from the Natural Sciences Foundation of Fujian Province (grant 2025J01486) and the Open Project Program of Key Laboratory of Discrete Mathematics with Applications of Ministry of Education (grant J20250601), Fuzhou University. The author thanks Zijian Deng for carefully reading this paper and finding an err in the proof of Lemma \ref{prism}. 

\end{document}